\documentclass[11pt]{amsart}

\usepackage{amsmath,amssymb,amscd,amsfonts,verbatim}
\usepackage{paralist}
\usepackage[mathscr]{eucal}
\usepackage{color}
\usepackage{comment}
\usepackage[pdftex]{hyperref}
\hypersetup{citecolor=blue,linktocpage}

\usepackage[colorinlistoftodos]{todonotes}

\newtheorem{thm}{Theorem}[section]
\newtheorem{lem}[thm]{Lemma}
\newtheorem{prop}[thm]{Proposition}

\newtheorem{assu-nota}[thm]{Assumption--Notation}
\newtheorem{claim}[thm]{Claim}
\theoremstyle{definition}

\newtheorem{rem}[thm]{Remark}
\newtheorem{ex}[thm]{Example}

\newcommand{\inv}{^{-1}}

\newcommand{\pp}{\mathbb P}

\DeclareMathOperator{\Pic}{Pic}
\DeclareMathOperator{\pic}{Pic}

\DeclareMathOperator{\Alb}{Alb}
\DeclareMathOperator{\albdim}{Albdim}

\DeclareMathOperator{\Imm}{Im}

\def\Kcal{{\mathcal K}}

\def\Mcal{{\mathcal M}}

\def\Ocal{{\mathcal O}}

\def\PP{{\mathbb P}}

\newcommand{\fie}{\varphi}

\numberwithin{equation}{section}

\title[The classification of minimal irregular surfaces\dots]{The classification of minimal irregular surfaces of general type with $K^2= 2p_g$}

\author{Ciro Ciliberto}
\address{Dipartimento di Matematica, II Universit\`a di Roma,
Italy}
\email{cilibert@axp.mat.uniroma2.it}
\author{Margarida Mendes Lopes}
\address{Departamento de Matem\'atica, Instituto Superior T\'ecnico, Universidade T\'ecnica de Lisboa Av. Rovisco Pais,
1049-001 Lisboa, Portugal}
\email{ mmlopes@math.ist.utl.pt}

\author{Rita Pardini}
\address{Dipartimento di Matematica, Universit\`a di Pisa\\
Largo B. Pontecorvo, 5,  56127 Pisa, Italy }
\email{pardini@dm.unipi.it}

\thanks{
  The first  and the third author are members of G.N.S.A.G.A.--I.N.d.A.M. The second  author is a member of the Center for Mathematical
Analysis, Geometry and Dynamical Systems (IST/UTL).  This research was partially supported by FCT (Portugal) through program POCTI/FEDER and
Project PTDC/MAT/099275/2008.}

\begin{document}
\begin{abstract} \medskip  Minimal irregular surfaces of general type satisfy $K^2\geq 2p_g$ (see \cite{debarre}). 
In this paper we  classify those surfaces  for which  the equality  $K^2=2p_g$ holds.  
\end{abstract}
\maketitle
\tableofcontents

\section{Introduction}

By a result of O. Debarre (see \cite{debarre}) a minimal irregular surface $S$ of general type satisfies $K^2\geqslant 2p_g$. 
In this paper we obtain the following classification for the case $K^2=p_g$:

\begin{thm}\label{classification} Let $S$ be a minimal complex surface of general type of irregularity $q>0$  satisfying $K^2=2p_g$.
Then $q\le 4$. 

If $q=1$, then the Albanese fibration is a genus 2 fibration with 2--connected fibres and $S$ is a double cover of a $\pp^1$-bundle over the elliptic curve  $Alb(S)$. 

If $q\geqslant 2$ then $\chi=1$ and: \\
\begin{inparaenum}
\item for $q=2$,  $S$ is the minimal desingularization of a double cover of a  principally polarized  abelian surface $(A,\Theta)$ branched on an effective divisor of class $2\Theta$ with at most negligible singularities;\\
\item for  $q=3$,  $S$ is the symmetric product of a curve of genus 3;\\
\item for $q=4$,  $S$ is the product of two curves of genus 2.
\end{inparaenum}
\end{thm}

The assertion $q\leqslant 4$ and the classification in the case $q=4$ are in  \cite[Th\'eor\`eme 6.3] {debarre}.  For $q=1$, $K^2=2p_g $ is the same as   $K^2=2\chi$ and these surfaces have been classified in \cite{hoV}.  Their Albanese pencil has 2--connected genus 2 fibres,  and the relative canonical map determines a 2--1 cover of a $\pp^1$-bundle over the elliptic curve ${\rm Alb}(S)$ (see  \cite[Theorem 5.2]{hoV} for a detailed description). 

To prove Theorem \ref {classification} we start  by completing  the classification  of such irregular surfaces with $\chi=1$ (see \S \ref {catanese}). The only yet unknown case was $p_g=q=2$ (see \S \ref{CL}  for a discussion of the known cases). In  Theorem \ref{thm:pgq2}, which is the main result in this paper,  we show that  such surfaces are exactly the so called \emph{Catanese surfaces}, described in  (ii) of Theorem \ref {classification}.  The case $K$  ample was proved by Manetti \cite {man}. Our proof does not need
this strong hypothesis and is, in our view, conceptually more transparent. It relies, as usual in these matters, on two main tools:
the paracanonical system and the Albanese map (see \S \ref {paracan}). However 
a new, essential ingredient  is the recent classification of curves $C$ such that a Brill--Noether locus $W^ s_d(C)$, strictly contained in the jacobian $J(C)$ of $C$, contains a variety $Z$ stable under translations by the elements of a positive dimensional abelian subvariety $A\subsetneq J(C)$ and such that $\dim(Z)=d-\dim(A)-2s$, i.e., the maximum possible dimension for such a $Z$ (see \S \ref {absub} and \cite{new}).

To complete the classification theorem one needs to rule out  the existence of minimal irregular  surfaces with $K^2=2p_g$, 
 $q\geqslant 2$, $\chi> 1$.  This is done in \S \ref{CL}.

\section{Preliminaries}\label{sec:symm}

\subsection{Abelian subvarieties of Theta divisors of Jacobians}\label{absub}

In \cite {new}, following \cite {ah, df}, we considered the following situation:\\

\begin{inparaenum}
\item  [(*)] $C$ is a smooth, projective, complex curve of genus $g$,  $Z$ is an irreducible $r$--dimensional subvariety of a Brill--Noether locus $W^ s_d(C)\subsetneq J^ d(C)$, and $Z$ is stable under translations by the elements of an abelian subvariety $A\subsetneq J(C)$ of dimension $a>0$ (if so, we will say that $Z$ is $A$--\emph{stable}).\\
\end{inparaenum}

Here $J^d(C)$ is the set of equivalence classes of divisors of degree $d$ on $C$ and $J(C):=J^0(C)$ is the \emph{Jacobian variety} of $C$.

It was proved in \cite {df}  that if (*) holds, then $r+a+2s\leqslant d$. In  \cite {new}, improving on partial results in \cite {ah, df},
we give the full classification of the cases in which (*) holds and $r+a+2s=d$. We will not need here the full strength of the results in \cite {new}, but only the part concerning the case $(d,s)=(g-1,0)$ (see \cite [Theorem 3.1]{new}). 

\begin{thm}\label{thm:main}
Let $C$ be a curve of genus $g$. Let $A\subsetneq J(C)$ be an abelian variety of dimension $a>0$ and $Z\subset W_{g-1}(C)$ an irreducible,  $A$--stable variety of dimension $r=g-1-a$.
Then there is a degree 2 morphism $\varphi\colon C\to C'$, with $C'$ smooth of genus $g'$, such that one of the following occurs:\\
\begin{inparaenum}
\item[(a)] $g'=a$, $A=\varphi^*(J(C'))$ and $Z=W_{g-1-2a}(C)+\varphi^*(J^a(C'))$;\\
\item[(b)] $g'=r+1$, $\varphi$ is \'etale, $A$ is  the  Prym variety of $\varphi$ and $Z\subset W_{g-1}(C)$ is the connected component of $\varphi_*\inv(K_{C'})$ consisting of  divisor classes $D$ with $h^0(C,\mathcal O_C(D))$ odd, where $\varphi_*\colon J^{g-1}(C)\to J^{g-1}(C')$ is the {norm map}. 
\end{inparaenum}

In particular, $Z\cong A$ is an abelian variety if and only if either we are in case (a) and $g=2a+1$, or  in case (b). 
\end{thm}

\subsection{Some generalities on irregular surfaces}\label{paracan}

Let $S$ be a smooth, irreducible surface.  We will use the standard notation $q(S)=h^0(S, \Omega^1_S)$,
$p_g(S)=h^0(S, \Omega^2_S)$,  $\chi(S)=\chi(\mathcal O_S)=p_g(S)-q(S)+1$ for the \emph{irregularity}, the \emph{geometric genus}, and the Euler characteristic of the structure sheaf. We may often use the simplified notation $p_g, q, \chi$. Numerical [resp. linear] equivalence will be denoted by $\sim$ [resp. by $\equiv$]. 

We will denote by $a\colon S\to {\rm Alb}(S)$ the \emph{Albanese morphism} of $S$. The dimension of $a(S)$ is denoted by ${\rm Albdim}(S)$ and is called the \emph{Albanese dimension} of $S$.  

A \emph{pencil of genus $b$} on $S$ is a morphism $f: S\dasharrow B$, with connected fibres, with $B$ a smooth  curve of genus $b$. The indeterminacy points of $f$ are the \emph{base points} of the pencil. If $b>0$, the pencil is said to be \emph{irrational} and it has no base points. 

As usual, $\vert K_S\vert$ (or simply $\vert K\vert$) denotes the \emph{canonical system} of $S$. If $\eta\in \pic^0(S)$, the linear system $\vert K+\eta\vert $ is called a {\it paracanonical
system} of $S$. A curve in $C_\eta\in \vert K+\eta\vert $ is a {\it paracanonical curve} on $S$. 

Assume $S$ is irregular of general type. We will denote by $\mathfrak K_S$ (or simply by 
$\mathfrak K$) the \emph{paracanonical system} of $S$, i.e., the Hilbert scheme of paracanonical curves on $S$. Note the morphism $p: \mathfrak K\to {\rm Pic}^0(S)$ acting as $C\mapsto \mathcal O_S(C-K)$. 
There is a unique component $\Kcal_S$ (or simply $\Kcal$) of $\mathfrak K$ dominating ${\rm Pic}^0(S)$ via $p$. It is called the 
{\it main paracanonical system} of $S$.

If ${\rm Albdim}(S)=2$ and
$\eta\in \pic^0(S)$ is general, one has $h^1(S,\eta)=0$ hence $\dim
(\vert K+\eta\vert) = \chi -1$ (this is the so--called \emph{generic vanishing theorem}, see \cite [Theorem 1] {gl}).
If $\eta\in \pic^0(S)$ and $C\in \vert K+\eta\vert$ are general,
then $C$ corresponds to the general point of ${\Kcal}$, which has dimension $q+\dim (\vert K+\eta\vert) = p_g$.

\section{Surfaces with $K^2=4$ and $p_g=q=2$}\label{catanese}

Let $S$ be minimal, of general type, with $p_g=q=2$.  One has 
$K^2\ge 4$ (see \cite {debarre}) and the equality is
attained in the following example.
 
 \begin{ex}\label{ex:catanese}
Let $(A,\Theta)$ be a principally polarized  abelian
surface. Let $p\colon S\to A$
be the double  cover branched on an effective, smooth divisor $B$ in the class of  $2\Theta$
so that $p_*\Ocal_S=\Ocal_A\oplus \theta^{-1}$,
 with $\theta$ in the class of $\Theta$ and $S$ is smooth and minimal. Then
$K=p^*(\theta)$ and  the invariants
of $S$ are $p_g=q=2$, $K^2=4$. 
The divisor $B$ may have   irrelevant singularities. In that case $S$ will be the minimal resolution of the double cover of $A$ branched on $B$. 
One has $\Alb(S)\cong A$. 

A special case is when $\Theta$ is \emph{reducible}, i.e. $A$ is the product of  
elliptic curves, and $\Theta=E_1+ E_2$, with $E_i$ elliptic curves
such that $E_1\cdot E_{2}=1$. In that case $S$ has two elliptic pencils $\mathcal M_1,\mathcal M _2$ of curves of genus $2$
(see \cite [Example 7.1] {CML}; according to  \cite {zuc}, this is the only case in which $S$ has an irrational pencil
of curves of genus 2: we will not use this result though). 

Note that the curves in $\mathcal K_S$ are the proper transforms via $p$ of the curves in the class of $\Theta$. If 
$\Theta$ is irreducible, then the general curve in $\mathcal K_S$ is smooth, otherwise the general curve in $\mathcal K_S$ is the sum of 
a curve in $\mathcal M_1$ plus a curve in $\mathcal M _2$.
\end{ex}

In this section we prove the following classification result.

\begin {thm}\label{thm:pgq2}  A minimal surface of general type
with $p_g=q=2$, $K^2=4$ is as in Example \ref {ex:catanese}.
\end{thm}

By \cite [Proposition 2.3]{CML},   minimal surfaces $S$ of general type with $K^2=4$ and  $p_g=q=2$  have ${\rm Albdim}(S)=2$, i.e.,  $a\colon  S\to {\rm Alb}(S)$ is surjective. 
Surfaces 
of general type with $p_g=q=2$ and with an irrational pencil   have been studied in \cite{pen0, pen, zuc}, but we will not use their results here.

\begin{lem}\label{lem:para-smooth} Let $S$ be a minimal
surface of general type with $p_g=q=2$, $K^2=4$. Then:\\
\begin{inparaenum}
\item either the general curve in the main paracanonical system $\Kcal$ is smooth,\\
\item or $S$ is as in Example \ref {ex:catanese}, with $\Theta$ reducible.
\end{inparaenum}
\end{lem}

\begin{proof} Let $C\in \Kcal$ be general. Write $C=F+M$ where $F$ is the fixed divisor of $\Kcal$.  

Suppose first  $M$  reducible.  
Since $\dim({\Kcal})=2$, $M$ must consist of two distinct
irreducible components $M_i$ each moving in a 1--dimensional family of curves $\mathcal M_i$, with $1\leqslant i\leqslant 2$. 
The index theorem yields
$K\cdot M_i\geqslant 2$, for $1\leqslant i\leqslant 2$. On the other hand  $4=K^2=K\cdot F+K\cdot M_1+K\cdot M_2$, 
hence $K\cdot F=0, K\cdot M_1=K\cdot M_2=2$. Then $M_1^2=M_2^2=0$ and ${\Mcal}_i$ is a pencil
of curves of genus $2$, for $1\leqslant i\leqslant 2$. Such a pencil is not rational by  \cite [Lemma on p. 345] {Be},
hence it is of genus 1, because ${\rm Albdim}(S)=2$. This implies $M_1\cdot F=M_2\cdot F=0$, hence $F=0$ by the $2-$connectedness of paracanonical divisors  and 
$M_1\cdot M_2=2$. 
 Let $f_i\colon  S\to E_i$ be the  elliptic pencils  ${\Mcal}_i$, for $1\leqslant i\leqslant 2$. The morphism $f=f_1\times f_2\colon  S\to E_1\times E_2$ is a double cover and we are in case 
(ii). 

Next we may assume $M$ irreducible, and  we prove that $C$ is irreducible. Indeed, the argument of   \cite[Lemma 4.1]{CML} shows that $M^ 2\geqslant 3$. 
Then $F\not=0$ yields  $K^2\geqslant K\cdot M=F\cdot M+M^2\geqslant 5$, a contradiction. 

Finally we prove that $C$ is smooth.  Assume, to the contrary,  $C$ singular. 
Let $c\in \Kcal$ be the point corresponding to $C$.
The $2$--dimensional  tangent space
$T_{\Kcal,c}$ is contained in $H^0(C,N_{C\vert S})\cong H^0(C,{\Ocal}_C(C))$. It therefore
corresponds to a $g^1_4$ on $C$. Every section of $H^0(C,{\Ocal
}_C(C))$ in $T_{\Kcal,c}$ vanishes at each singular point of $C$. This implies that $C$
has a unique singular point $x_C$. Consider the rational 
map $\pic^0(S)\dasharrow S$ which associates to $C$ its singular point $x_C$. 
This map is not dominant. This means that for $C$ general in ${\Kcal}$ there is a 
$1$--dimensional system ${\Kcal }_C$ of curves in ${\Kcal}$ sharing with $C$ the singular point $x_C$. Since
the curves in ${\Kcal}_C$ have no variable intersection
off $x_C$, then ${\Kcal}_C$ is a pencil, and it is rational, because it has the base point
$x_C$. Thus the surface $P$ parametrizing ${\Kcal}$ would be ruled, a contradiction, since $P$ is birational to ${\rm Pic}^0(S)$.
 \end{proof}
 
By Lemma \ref {lem:para-smooth} we may assume from now on that the general curve  $C\in \Kcal$ is smooth of genus 5. 
Consider the \emph{restriction morphism} $\mathfrak r\colon \Pic^0(S)\to J^4(C)$ acting as  $\eta\mapsto  \mathcal O_C(C+\eta)$. This map is injective by \cite[Proposition 1.6]{cfm}.  By generic vanishing,   if $\eta\in \Pic^0(S)$ is general, one has $h^0(C,O_C(C+\eta))=h^0(C,O_S(C+\eta))=1$. Hence the image $A$ of $\mathfrak r$ is an abelian surface contained in $W_4(C)$.

\begin{lem}\label{lem:jacobian} In the above set up there exists a smooth genus $2$ curve $C'$ and a degree 2 morphism $\varphi\colon C\to C'$ such that $A=\varphi^*(J^2(C'))$.
\end{lem} 
\begin{proof}
We can apply Theorem \ref{thm:main} with $g=5$ and $r=a=2$, hence it suffices to show that case (b) of that theorem does not occur. 
To prove this, we make the following remarks:\\
\begin{inparaenum}
\item  [(i)]  by generic vanishing, the long exact sequence in cohomology of
\[
0\to \mathcal O_S(K_S-C)\to  \mathcal O_S(K_S)\to  \mathcal O_C(K_S)\to 0
\]
determines isomorphisms
\begin{equation}\label{eq:iso}
H^i(S,\mathcal O_S(K_S))\cong H^i(C,  \mathcal O_C(K_S)),\,\, {\rm for} \,\, 0\leqslant i\leqslant 1;
\end{equation}
\item [(ii)] by the analysis in \cite [Proof of Proposition 4, p.  155] {beauville-annulation} (which holds because $q=2$ is even),  if $s\in H^0(S, O_S(K_S))$ is general, then 
\[
H^1(S, \mathcal O_S) \buildrel  {\cup s}\over \longrightarrow H^1(S, O_S(K_S))
\]
is an isomorphism;\\
\item [(iii)] hence, if we denote by $T$ the image of $H^1(S,\mathcal O_S)$ in $H^1(C, \mathcal O_C)$ under the obvious, injective restriction map,  
and if  $s\in H^0(C, O_C(K_S))$ is general, then 
\begin{equation}\label{eq:iso1}
T \buildrel  {\cup s}\over \longrightarrow H^1(C, O_C(K_S))
\end{equation}
is an isomorphism;\\
\item [(iv)] dualizing \eqref {eq:iso} one has isomorphisms
\begin{equation}\label{eq:iso2}
H^i(C,\mathcal O_C(C))\cong H^{i+1}(S,  \mathcal O_S),\,\, {\rm for} \,\, 0\leqslant i\leqslant 1.
\end{equation}
Hence for every non--zero $s\in H^0(C,\mathcal O_C(C))$ there exists a non--zero $v\in T$ such that $v\cup s=0$:  $v$ is the element in $T\cong H^{1}(S,  \mathcal O_S)$ corresponding to $s$ in the isomorphism in \eqref {eq:iso2} for $i=0$.  Thus in this case 
\begin{equation}\label{eq:iso3}
T \buildrel  {\cup s}\over \longrightarrow H^1(C, O_C(C))
\end{equation}
is not an isomorphism.
\end{inparaenum}

Suppose, by contradiction,  we are in the Prym case (b) of Theorem \ref{thm:main} and denote by $\iota$ the corresponding fixed point free   involution of $C$.
 Then $\iota$ on $C$ acts on the Prym 
variety $A\subset J^{4}(C)$ as $\mathcal O_C(D)\mapsto \mathcal O_C(K_C-D)$. Hence $\iota$ 
interchanges $\mathcal O_C(C)$ and $\mathcal O_C(K_S)$ and it acts on $T\cong H^{1}(S,  \mathcal O_S)$, which is the tangent space to $A$, as multiplication by 
$-1$. This contradicts the fact that \eqref {eq:iso1} is an isomorphism whereas \eqref {eq:iso3} is not. \end{proof}

\begin{rem}\label{rem:g} It is perhaps useful to briefly explain the geometric idea underlying the proof of Lemma \ref {lem:jacobian}.

Consider the surface $P$ parametrizing ${\Kcal}$. It is the blow--up of ${\rm Pic}^0(S)$
at those, finitely many, points $\eta$ such that the linear system $\vert K+\eta\vert$ is \emph{superabundant} (i.e., $h^1(S,\mathcal O_S(K+\eta))>0$;  actually one proves that $h^1(S,\mathcal O_S(K+\eta))\leqslant 1$ for all $\eta\in {\rm Pic}^0(S)$) but not \emph{exuberant} (which means that $\vert K+\eta\vert\subsetneq \Kcal$, in which case $\vert K+\eta\vert\cap \Kcal$ is a single curve). By \cite[Proposition 4]  {beauville-annulation}, the canonical system $\vert K\vert$ is exuberant, so there is a unique canonical curve $C_0\in \Kcal$. Accordingly we do not need to blow up ${\rm Pic}^0(S)$ at $0$ in order to obtain $P$. 

Consider the \emph{restriction map} $\mathfrak s\colon P\dasharrow C(4)$, with  target the 4--tuple symmetric product of $C$. This map is not defined at the point $c$ corresponding to $C$ and it maps a point $c'$ corresponding to a curve $C'\in \mathcal K$ different from $C$ to the degree 4 divisor  cut out on $C$ by $C'$.
In order to resolve the indeterminacy of $\mathfrak s$, one has to blow--up  $c$, thus obtaining a new surface $P'$ with the exceptional divisor $E$ corresponding to $c$. One has $E\cong \vert \mathcal O_C(C)\vert$ (see the proof of Lemma \ref {lem:para-smooth}), and $\mathfrak s$ extends naturally to the points of $E$.  We denote by $\mathfrak s'\colon P'\to C(4)$ the extension of $\mathfrak s$, which is injective by \cite[Proposition 1.6]{cfm}. We abuse notation and identify $P'$ with  ${\rm Im}(\mathfrak s')$.

Apply now Theorem \ref {thm:main}. We have a degree 2 morphism $\varphi: C\to C'$ and  assume we are in case (b). Let $\iota$
 be the involution on $C$ determined by $\varphi$. This involution acts on $C(4)$ and fixes $P'$,  acting on it  in a non--trivial way. So  $\sigma(E)\neq E$ should be a curve on $P'$. On the other hand
a general point on $E$ corresponds to a general divisor $D\in  \vert\mathcal O_C(C)\vert$ and $\sigma(D)$ is  a general divisor  in $\vert\mathcal O_C(K_S)\vert$. 
This is a contradiction because, as we saw, $P'$ contains only one point $c_0$ corresponding to a canonical curve $C_0\in \vert K_S\vert$, hence ${\rm Im}(\mathfrak s')$
contains only the divisor $D_0\in \vert\mathcal O_C(K_S)\vert$ cut out by $C_0$ on $C$. So case (b) does not occur and we are in case (a). 
\end{rem}

\begin{proof}[End of proof of Theorem \ref{thm:pgq2}] By Lemma \ref {lem:para-smooth}, we may assume $C\in \Kcal$ general to be smooth. By Lemma \ref{lem:jacobian}, there is a degree $2$ morphism $\varphi\colon C \to C'$, with $C'$ of genus 2,  such that $A=\varphi^*(J^2(C'))$. So $\Pic^0(S)$ is a  principally polarized abelian surface and we may identify both $\Pic^0(S)$ and $\Alb(S)$ with $A$. 

We claim that there is an involution $\mathfrak i$ of $S$ such that $\mathfrak i (C)=C$ and $\mathfrak  i_{\vert C}$ coincides with the involution $\iota$ determined by the double cover $\varphi\colon C \to C'$.  Indeed,  Lemma \ref{lem:jacobian} implies that, if $x\in C$ is general and $D\in \Kcal$ is any curve containing $x$, then $\iota(x)\in D$. Hence $\iota(x)$ is independent of $C$, proving our claim. Note that $\mathfrak i$ is biregular, since $S$ is minimal of general type. 

Consider the minimal desingularization $X$ of the quotient $S/\mathfrak i$.
It contains a 2-dimensional family of (generically smooth and irreducible) curves $\Gamma$ of genus $2$, with $\Gamma^ 2=2$, i.e. the genus 2 quotients of the curves in $\Kcal_S$ under the involution $\mathfrak i$. By \cite [Theorem 0.20]{ccm},
 all curves $\Gamma$ are isomorphic and $X$ is birational to the symmetric product $\Gamma(2)$, which is birational to $J(\Gamma)\cong A$. So we have a birational morphism $X\dasharrow  A$, hence we have a degree 2 map $S\dasharrow A=\Alb(S)$, which is the Albanese morphism.  Note that via $X\dasharrow A$, the curves $\Gamma$ map to the $\Theta$ divisors. 

Let $B$ be the branch divisor of $a\colon  S\to A$. Since the curves in $\Kcal$ have genus 5, we have $B\cdot \Theta=4$. By the index theorem we have $B^2\leqslant 8$. On the other hand $B$ is divisible by $2$ in $NS(A)$, hence $B^2\geqslant 8$. In conclusion $B^2=8$ and the index theorem implies that the class of $B$ is $2\Theta$ and the assertion follows. 
\end{proof}

\begin{rem} Let $S$ be as in Example \ref {ex:catanese} with $B$ smooth.
 Then the ramification curve $R$ is a canonical curve isomorphic to $B$, hence  smooth. So the general curve $C\in \vert K\vert$ is smooth and, as above,  we can consider the restriction map ${\rm Pic}^0(S)\to W_4(C)$, whose image is an abelian surface $A\subset W_4(C)$.  We can apply Theorem \ref{thm:main}, but now Lemma  \ref{lem:jacobian} does not hold, and in this situation the Prym case (b), and not case (a),  of Theorem  \ref{thm:main} occurs. Indeed, the bicanonical morphism  is not birational for $S$: it is in fact composed with an involution $\sigma$ such that $\Sigma=S/\sigma$ is a surface with $20$ nodes (over which the double cover $\phi: S\to \Sigma$ is ramified). If $X$ is the  minimal desingularization of $\Sigma$, one has $p_g(X)=2, q(X)=0, K^2_X=2$ 
(see \cite {CML}). So the general curve $C\in \vert K_S\vert $ is the \'etale double cover of the general curve $C'\in \vert K_X\vert$. One has the following commutative diagram
\[
 \renewcommand{\arraystretch}{1.3}
 \begin{array}{ccc}
 \hphantom{\scriptstyle{g'}\ } S &\buildrel  \phi\over \longrightarrow{} & \Sigma
 \hphantom{\ \scriptstyle{g}} \\
 {\scriptstyle{p}\ }\big\downarrow && \big\downarrow \ \scriptstyle{q} \\
 \hphantom{\scriptstyle{q}\ } A & \buildrel \mathfrak \kappa\over \longrightarrow{} & Y
 \hphantom{\ \scriptstyle{p}}
\end{array}
\]
where $Y\subset \PP^3$ is the 16--nodal Kummer surface of $A$ ($\kappa$ is the obvious double cover), and $q: \Sigma\to Y$ is a double cover branched along a  smooth plane section $H$ of $Y$ (which pulls back via $\kappa$ to the branch divisor of $a$), plus six nodes lying on a conic $\Gamma\subset Y$. 

There is a unique curve $C_0$ in the intersection of $ \vert K_S\vert $ with $\Kcal$ (see Remark \ref {rem:g}), i.e., the proper transform of $\Gamma$ on $S$. It is interesting to notice that for $C_0$ both cases (a) and (b) of Theorem  \ref{thm:main} occur at the same time (see \cite [Remark 3.2]{new}).

We finally notice that the same idea of proof of Theorem \ref {thm:pgq2} can be applied to recover the classification of minimal surfaces $S$ with $p_g=q=3$ (see \cite  [Proposition (3.22)]  {ccm}). In this case $q$ is odd, the analogue of Lemma  \ref{lem:jacobian} does not hold, and the Prym case (b) of Theorem  \ref{thm:main} occurs. We do not dwell on this here.  \end{rem}

 \section{Proof of Theorem \ref{classification}}\label{CL}

Let $S$ be a minimal complex surface of general type of irregularity $q>0$  satisfying $K_S^2=2p_g$.
By \cite[Th\'eor\`eme 6.3] {debarre} one has $q\le 4$ and $q=4$ if and only if $S$ is the product of two curves of genus 2 (see also (\cite {Be}).
If $q=1$, then $K_S^2=2\chi$: these surfaces have described in \cite{hoV} (see also \cite{Ca} for the case $q=\chi=1$). 
So we may assume $2\leqslant q\leqslant 3$ and

\begin{itemize}
\item $q=2$: $S$ is as in Example \ref {ex:catanese} (see Theorem \ref{thm:pgq2});
\item $q=3$: $S$ is the symmetric product of a curve of genus 3 (\cite [Proposition (3.22)]  {ccm}).
\end{itemize}

So to prove Theorem \ref{classification}, it suffices to show that:

\begin{prop} \label{prop:no-chi2}
There is no minimal surface $S$  with  $2\le q\le 3$, $\chi\ge 2$ and $K^2=2p_g$.

\end{prop}

\begin{proof}
Suppose by contradiction that such a surface $S$ exists.

\begin {claim}\label{cl:1} 
One has $\albdim (S)=2$, $q=3$, $p_g=4$.
\end{claim}

\begin{proof}[Proof of the Claim] Suppose, by contradiction, that  $\albdim (S)=1$ and  let   $a\colon S\to B$  be the Albanese pencil, with $B$ of genus $q$. 
We denote by $g$ the genus of the fibres of $a$  and one has $g\geqslant 2$  because  $S$ is of general type,  

By  \cite[Theorem 3.1]{hoV}, if $K^2<3\chi$,
 the general fibre of $a$ is hyperelliptic of genus  $g=2$  or $3$. In this case  \cite[Theorem 2.1]{hoV} (the \emph{slope inequality} for hyperelliptic fibrations) applies and gives $q=1$, a contradiction. 

If $K^2\geqslant 3\chi$, then   $2p_g=K^2\geqslant 3p_g-3q+3$ yields $p_g\leqslant 6$ for $q=3$ and $p_g\leqslant 3$   for  $q=2$.   Since $K^2=2p_g$, one has  $K^2\leqslant 12$, if $q=3$ and $K^2\leqslant 6$ if $q=2$.  But, by Arakelov's theorem (see \cite{Be}), we have   $K^2\geqslant 8(g-1)(q-1)\geqslant 8(q-1)$,
  a contradiction.

So $\albdim (S)=2$.  Then the Severi's inequality  proved in \cite{pa}  gives
$2p_g=K^2\geqslant  4\chi$, 
i.e., $p_g\le 2q-2$. Since $\chi\ge 2$ and $q\le 3$, the only possibility is $p_g=4$, $q=3$.
\end{proof}

\begin {claim}\label{cl:2} 
$S$ has no  irrational pencil  $f\colon S\to B$ with $B$ of genus $b>1$. Therefore:\\
\begin{inparaenum}[(i)]
\item   the map $\Phi\colon\bigwedge^2H^0(S, \Omega^1_S)\to H^0(S, \mathcal O_S(K_S))$ is injective;\\
\item  the Albanese image $\Xi$ of $S$ is not covered by elliptic curves.
\end{inparaenum}
\end{claim}

\begin{proof}[Proof of the Claim] 
 Let $g$ be the genus of the general fibre of $f$. By Arakelov's theorem (see \cite{Be}) we have $8=K^2\geqslant 8(g-1)(b-1)$, hence $g=b=2$.
By  \cite[Theorem 3] {hop}, one has $K^2\geqslant  2\chi-6+6b=10$, a contradiction. This proves the first assertion and (i) follows by the  Castelnuovo--de Franchis' theorem. 

To prove (ii), suppose  $\Xi$ is covered by elliptic curves. Then there is an elliptic curve $E\subset {\rm Alb}(S)$ such that
the image of $S$ via the composition of the Albanese morphism  and the morphism ${\rm Alb}(S)\to {\rm Alb}(S)/E$ is a curve $B$. Since $B$ has to span ${\rm Alb}(S)/E$, which has dimension 2, then $B$ has genus $b\geqslant 2$, contradicting the first assertion. \end{proof}

\begin {claim}\label{cl:3} 
Let $F$ be the general fibre of a pencil of $f: S\dasharrow B$ with $B$ of genus $b$ (possibly $b=0$ and the pencil is linear, with base points), and let $F$ be irreducible of  geometric genus $g$.
  Then: \\
  \begin{inparaenum}  [(i)]
  \item $K\cdot F\geqslant 4$ and $g\geqslant 3$;\\
  \item if $K\cdot F=4$, then either  $F^2=0,  g=3, b=1$ or $F^2=2, b=0, g=4$ and $K\sim 2F$.
 \end{inparaenum}
 
   In particular $S$ has no pencil of curves of genus 2 and  its bicanonical map is a birational morphism.
\end{claim}

\begin{proof}[Proof of the Claim] 
 By Claim  \ref {cl:2}, we have $b\leqslant 1$. By blowing up if necessary,  we can assume the pencil has no base points and $F$ is smooth of genus $g$. 
 Then, by \cite[Lemme on p. 345] {Be},  one has $3=q\leqslant b+g$ with equality only if  $S$ is birational to a product of curves of genus $b$ and $g$. Since $S$ is of general type  and $b\leqslant 1$, equality cannot hold and thus either $b=1$ and $g\geqslant 3$ or $b=0$ and $g\geqslant 4$.

Consider now again  the original surface $S$.

If  $K\cdot F<4$,  by the index theorem and  parity, one either has  $F^2=0, K\cdot F=2$ or $F^2=1, K\cdot F=3$. The former case implies $g=2$, a contradiction.
The latter case gives $g=3$ and $b=0$ because the pencil has a base point, a contradiction again.  Hence $K\cdot F\geqslant 4$. If equality holds, again by the index theorem and parity either $F^2=0$ or $F^2=2$ and $K_S\sim 2F$.  In the former case $g=3$ and thus $b=1$.  

The last assertion follows from  \cite{ccm}. \end{proof}

Finally  we prove  Proposition \ref{prop:no-chi2} by showing that  surfaces $S$ as in Claim \ref {cl:1} do not exist. Suppose otherwise.
Consider the map $\phi\colon S\dasharrow \pp^2$ associated with the linear system $|\Imm (\Phi)|$, (see Claim \ref {cl:2}, (i)). This is the composition of the Albanese map $a\colon  S\to \Xi$ with the   Gauss map $\gamma\colon \Xi \dasharrow \pp^2$. Since $\Xi$ is not covered by elliptic curves,  $\gamma$ is dominant, hence a fortiori the image of the 
canonical map $\fie\colon S\dasharrow \pp^3$ is  a surface $\Sigma$. 
 
 Consider the multiplication map 
 \[\mu\colon {\rm Sym}^2(H^0(S, \mathcal O_S(K)))\to H^0(S,\mathcal O_S(2K)).\] Since 
 \[\dim ({\rm Sym}^2(H^0(S, \mathcal O_S(K))))=h^0(S,\mathcal O_S(2K))=10,\] we have two possibilities:\\
  \begin{inparaenum}
 \item[(i)] $\mu$ is an isomorphism:  since the bicanonical map of $S$ is a birational morphism (see Claim \ref {cl:3}), then 
  $\fie$ is also a birational morphism;\\
  \item[(ii)] $\dim( \ker (\mu)) =1$: in this case $\Sigma$  is a quadric.
  \end{inparaenum} 
  
  Since $q$ is odd and $\albdim (S)=2$, the canonical system is contained in the main paracanonical system $\Kcal$  (see  \cite{paracanonical}).  Hence if the general canonical curve is irreducible (smooth) then the general curve in $\Kcal$ is also irreducible (smooth).

Assume we are in case (i). Let $C\in \Kcal$ be  general, which is smooth by the above remark.  For any paracanonical curve $D\in |2K-C|$ there exists a quadric  $Q_D$ of $\pp^3$ such that $D+C$ is the divisor of $\fie^*(Q_D)$.  Since $h^0(S, \mathcal O_S(D))=\chi=2$,   there exist at least two distinct quadrics $Q_1$ and $Q_2$ of $\pp^3$ containing $\fie(C)$. Since $\fie(C)$ is irreducible and non degenerate of degree 8, this is a contradiction. So  case (i) does not occur. 

Assume now we are in case (ii) and suppose first $\Sigma$ is non--singular. The two line rulings $|L_1|$, $|L_2|$ of $\Sigma$ determine two  pencils on $S$ with general fibres $F_1$, $F_2$. For $1\leqslant i \leqslant 2$ the strict transform of a general element of $\vert L_i\vert$ is numerically equivalent to $r_iF_i$, with $r_i\geqslant  1$, and $K-(r_1F_1+r_2F_2)$ is numerically equivalent to an effective divisor. Hence we have $8=K^2\ge r_1K\cdot F_1+r_2K\cdot F_2\ge 4(r_1+r_2)$ (the last inequality follows by Claim \ref {cl:3}). So $r_1=r_2=1$, $K\cdot F_1=K\cdot F_2=4$ and, by  Claim \ref {cl:3},  $|F_1\vert$ and $|F_2|$ are distinct linear pencils with $F_1^2=F_2^2=F_1\cdot F_2=2$ (in particular $F_1\sim F_2$ and $K\sim 2F_1\sim 2F_2$). The pencil $|F_i\vert$ has a base scheme $\beta_i$ of lenght 2, hence $F_i$ is smooth of genus 4, for $1\leqslant i \leqslant 2$.

Set $F=F_i$ (with $i=1$ or $2$). The restriction sequence
\[
0\to \mathcal O_S(K_S)\to O_S(K_S+F)\to O_F(K_F)\to 0
\]
and $h^1(S,\mathcal O_S(K_S+F))=0$ yield the long exact sequence in cohomology
$$0\to H^0(S,\mathcal O_S(K_S))\to H^0(S,\mathcal O_S(K_S+F))\buildrel  {r}\over \longrightarrow H^0(F,\mathcal O_F(K_F))\to H^1(S,\mathcal O_S(K_S))\to 0$$
This implies that $\dim({\rm Im}(r))=1$, i.e.,  the rational map determined by  $|K_S+F|$ contracts the general curve in $\vert F\vert$ to a point. This is a contradiction, because the canonical image of $S$ is a surface. 

Finally, assume  $\Sigma$ is a quadric cone. The same arguments as before show that the line ruling of $\Sigma$ determines a pencil $|F|$ on $S$ with $K\cdot F=4$, $F^2=2$. Then, as above,  the rational map determined by $|K_S+F|$ contracts the general curve of $\vert F\vert$ to a point, again a contradiction, which shows that case (ii)  cannot occur either. This finishes the proof of the Proposition. \end{proof}

\begin{rem} In order to exclude the existence of surfaces $S$ as in Claim \ref {cl:1} we could have argued, in principle,  as in the proof of Theorem \ref{thm:pgq2}, i.e. we could have used the results in \cite {new}. To do so, one should first prove that the general curve in $\Kcal_S$ is smooth. This is not shorter, actually it is a bit more involved, than the proof presented here.\end{rem}

     \end{document}